\documentclass[12pt]{amsart}

\usepackage[]{graphicx,color,verbatim}
\usepackage[colorlinks=false]{hyperref}
\usepackage[paperwidth=8.5in,paperheight=11in,inner=2.5cm,outer=2.5cm,top=2.5cm,bottom=2.5cm]{geometry}
\usepackage{MnSymbol}

\newtheorem{lemma}{Lemma}[section]
\newtheorem{theorem}[lemma]{Theorem}

\newtheorem*{corollary*}{Corollary}
\newtheorem{cor}[lemma]{Corollary}
\newtheorem{proposition}[lemma]{Proposition}
\newtheorem*{proposition*}{Proposition}

\newtheorem*{thm*}{Theorem}

\theoremstyle{definition}

\newtheorem{remark}[lemma]{Remark}
\newtheorem*{remark*}{Remark}
\newtheorem{rem}[lemma]{Remark}
\newtheorem*{rem*}{Remark}

\newtheorem{example}[lemma]{Example}

\newtheorem*{quest*}{Question}

\theoremstyle{definition}
\newtheorem{exm}[lemma]{Example}
\newtheorem{Def}[lemma]{Definition}
\newtheorem*{def*}{Definition}

\newcommand\NN{\mathbb{N}}

\newcommand\RR{\mathbb{R}}

\newcommand{\CV}{\ensuremath{\mathcal{CV}}}

\begin{document}

\title{The core variety and representing measures in the truncated moment problem}
\author{Grigoriy Blekherman} 
\address{School of Mathematics, Georgia Institute of Technology, Atlanta, GA 30332}
\email{greg@math.gatech.edu}
\author{Lawrence Fialkow }
\address{Department of Computer Science, State University of New York, New Paltz, NY 12561}
\email{fialkowl@newpaltz.edu}


\begin{abstract}

The classical Truncated Moment problem asks for necessary and sufficient conditions so that a linear functional $L$ on $\mathcal{P}_{d}$, the vector space of real $n$-variable polynomials of degree at most $d$, can be written as integration with respect to a positive Borel measure $\mu$ on $\RR^n$. We work in a more general setting, where $L$ is a linear functional acting on a finite dimensional
vector space $V$ of Borel-measurable functions defined on a $T_{1}$ topological space $S$. Using an iterative geometric
construction, we associate to $L$ a subset of $S$ called the \textit{core variety},
$\mathcal{CV}(L)$. Our main result is that $L$ has a representing measure $\mu$ if and only if
$\mathcal{CV}(L)$ is nonempty. In this case, $L$ has a finitely atomic representing measure,
and the union of the supports of such measures is precisely $\mathcal{CV}(L)$. We also use the core variety to describe the facial decomposition of the
cone of functionals in the dual space $V^{*}$ having representing measures.
We prove a generalization of the Truncated Riesz-Haviland Theorem of \cite{CF2}, which permits us
to solve a generalized Truncated Moment Problem in terms of positive extensions of $L$.
These results are adapted to derive a Riesz-Haviland Theorem for a generalized Full Moment Problem
and to obtain a core variety theorem for the latter problem.

\end{abstract}

\maketitle

\section{Introduction}

Let $S$ be a (nonempty) topological space in which points are closed, and let $V$ be a finite dimensional vector space of Borel-measurable real-valued functions on $S$. Let $V^*$ be the dual space of linear functionals on $V$. The main question of this paper is the following version of the generalized truncated moment problem: given $L \in V^*$, does $L$ have a \textit{representing measure},
 i.e., a positive Borel measure $\mu$ on $S$ such that $L(f)=\int_S f \, d\mu$ for all $f \in V$?

To motivate our results, we first recall the classical full and truncated moment problems. 
Let $\mathcal{P} \equiv \mathbb{R}[x_{1},\ldots,x_{n}]$ denote the ring of real polynomials in $n$ real variables. 
 For an $n$-dimensional
multisequence $\beta \equiv \beta^{(\infty)} := \{\beta_{i}\}_{i\in \mathbb{Z}_{+}^{n}}$,
$\beta_{0} > 0$, let $L \equiv L_{\beta}:\mathcal{P} \rightarrow \mathbb{R}$ 
be the \textit{Riesz linear functional} defined by $L(x^{i}) := \beta_{i}$ ($i\in \mathbb{Z}_{+}^{n}$),
where $x \equiv (x_{1},\ldots, x_{n}) \in 
\mathbb{R}^{n}$, $i = (i_{1},\ldots, i_{n})\in \mathbb{Z}_{+}^{n}$, and
$x^{i} := x_{1}^{i_{1}}\cdots x_{n}^{i_{n}}$.
Given $\beta$ and a closed set $K\subseteq \mathbb{R}^{n}$, the classical \textit{Full 
$K$-Moment Problem} seeks necessary and sufficient conditions for the existence
of a positive Borel measure $\mu$ on $\mathbb{R}^{n}$ satisfying $\operatorname{supp} \mu \subseteq K$ and
$\beta_{i} = \int_{\mathbb{R}^{n}} x^{i} d\mu$, $i\in \mathbb{Z}_{+}^{n}$.
 The theorems of M. Riesz \cite{Rie}
($n=1$) and %of 
E.K. Haviland \cite{Hav} ($n>1$) show that $\beta$ has a representing measure 
supported in  $K$ if and only if $L$ is \textit{positive with respect to K}, i.e., $L(p) \ge 0$ whenever $p\in \mathcal{P}$
satisfies $p|_K \ge 0$. For general $K$, the Riesz-Haviland Theorem is difficult to apply,
but various concrete applications are known in special cases (cf. \cite{A, AK, Las,Lau}).
We note particularly Schm\"{u}dgen's solution to 
the Full $K$-Moment Problem  in the case
when $K$ is a compact basic semi-algebraic set \cite{Sch1};
Schm\"{u}dgen's work  established a close connection between moment theory
and real algebraic geometry.

Now let $\mathcal{P}_{d}$ denote the polynomials of degree at most $d$ in $n$ real variables,
and let $\beta \equiv
\beta^{(d)}$ denote an $n$-dimensional multisequence of degree $d$,
i.e., $\beta = \{\beta_{i}\}_{i\in \mathbb{Z}_{+}^{n},~ |i|\le d}$. The classical
\textit{Truncated $K$-Moment Problem} (TKMP) seeks conditions for $\beta$ to have a representing
measure supported in $K$. Although the direct analogue of the Riesz-Haviland Theorem holds for TKMP when $K$ is compact (this is essentially the content of the proof of Tchakaloff's Theorem \cite{Tch}), such is not the case for general $K$. Instead, the Truncated Riesz-Haviland Theorem \cite{CF2} shows that $\beta \equiv \beta^{(2d-2)}$ or $\beta \equiv \beta^{(2d-1)}$
admits a representing measure supported in a closed set $K$ if and only if the corresponding
Riesz functional $L_{\beta}$ can be extended to a linear functional $\widehat{L}:\mathcal{P}_{2d+2} \rightarrow \mathbb{R}$ positive with respect to $K$. Alternately, in \cite{CF1}, Curto and the second-named author formulated a solution in terms of \textit{flat extensions}
of moment matrices.

As in the 
Full $K$-Moment Problem, the preceding solutions to the truncated problem are
difficult to apply in general, although various concrete solutions  are known
in special cases \cite{F1}. Recently, the first-named author \cite{Ble} used an approach
based in 
convex geometry to establish the following concrete condition:  $\beta^{(2d-1)}$
has a representing measure whenever the moment matrix $M$ corresponding to $\beta^{(2d)}$ satisfies
$\operatorname{rank} M \le 3d-3$.

In view of the difficulty of applying the  extension results of \cite{CF1} and \cite{CF2}
to general multisequences $\beta \equiv \beta^{(d)}$,  in \cite{F2} we
introduced an alternate approach based on a geometric invariant called
the \textit{core variety} of $L_{\beta}$. In \cite{F2}, the core variety was used
to establish the existence or nonexistence of representing measures in a
number of special cases. The core variety contains the support of each representing measure
for $\beta$, but what emerged from \cite{F2} was the following question:
\begin{quest*} If the core variety of $\beta$ is nonempty, does $\beta$ have
a representing measure?
\end{quest*}

In the present note, we resolve this question affirmatively, and in a very general setting. 
We also prove an
extended version of the Truncated Riesz-Haviland Theorem, and apply it to the Full Moment Problem to obtain a generalized version of the Riesz-Haviland Theorem.
This  also entails extending  Stochel's Theorem \cite{Sto}, concerning the connection
between the full and truncated moment problems. As a consequence of these results,
we also obtain an analogue of our core variety results for the Full Moment Problem.
Our approach emphasizes convex-geometric aspects of the truncated moment problem, and we use a number of elementary results from finite-dimensional convex geometry. 

\begin{remark*}

Finitely atomic measures correspond to positive linear combinations of point evaluation functionals in $V^*$.
 As we explain in more detail below (cf. Remark \ref{rem:farm}), our results concerning finitely atomic representing measures, when expressed equivalently in terms of conical combinations of point evaluations, 
actually hold in complete generality, with virtually no assumptions on $S$ or on the functions in $V$. However, by working in complete generality,  the connection to the classical moment problem is not as transparent as in
our setting of a topological space $S$ with points closed, where all point evaluation functionals have representing measures.

\end{remark*}

 In the case where the linear functional $L$ is known to come from a measure, the concept of core variety was studied by di Dio and Schm\"{u}dgen in \cite{DS, Sch2}. In particular, they independently obtained Theorem \ref{thm:main} in case $L$ is known to come from a measure, and also Proposition \ref{prop:faces} \cite[Theorem 29]{DS}. Moreover, Gabardo did some work in this direction in the context of trigonometric moment problems in \cite{Gab}. Note that the setting for our results includes a variation of TKMP in which the sequence $\beta^{(d)}$ is only partially defined, i.e., moments corresponding to certain monomials are not prescribed. A special case of this problem has been studied by Laurent and Mourrain \cite{LM}, and a somewhat analogous partially defined trigonometric moment problem appears in \cite{Gab}. \\

\noindent \textit{Acknowledgements.} The first author was partially supported by NSF grant DMS-1352073. Both authors are grateful to organizers of Oberwolfach Workshop on Real Algebraic Geometry With a View Toward Moment Problems and Optimization (March, 2017), where fruitful discussions took place.

\subsection{Main Results}\label{sec:results}

With $S$ and $V$ as  above, let $P \equiv P_{V}$ be the set of all functions in $V$ which are nonnegative on $S$:
$$P=\{f \in V \,\,\ \mid \,\, f(s) \geq 0 \,\,\, \text{for all} \,\,\, s \in S \}.$$
Since $V$ is finite dimensional, we may equip $V$ with the Euclidean topology; note that $P$ is then a closed convex cone in $V$. It also follows that $P$ is a \textit{pointed} cone, i.e., $P$ does not contain lines.  We will assume that $P$ contains a strictly positive function. This assumption is important for our main results; see Remark \ref{rem:pos} for additional discussion, and Remark \ref{rem:gen} for a possible generalization. For a subset $T$ of $V$ we will denote by $\mathcal{Z}(T)\subset S$ the set of common zeros of functions in $T$:
$$\mathcal{Z}(T)=\{s \in S \, \mid \, f(s)=0 \,\,\, \text{for all $f \in T$}\}.$$
Let $M$ be the cone of functionals in $V^*$ which have representing measures. We observe that $M$ is a convex cone in $V^*$, but $M$ may fail to be closed \cite{EF}. Our main goal is to characterize linear functionals belonging to $M$.

Given a linear functional $L \in V^*$,
let $S_0 \equiv S_0[L]=S$, and let $S_1\equiv S_1[L]$ be the zero set of all nonnegative functions in the kernel of $L$:

$$S_1=\mathcal{Z}(p\in P \,\, \mid \,\, L(p)=0).$$%%%%%%%%%%%%%%%%%%%%%%%%%%%%%%$$

\noindent Then we iteratively define: $$S_{i+1} 
%%%%%%%%%%%%%%%%%%%%%%%%%%%%%%%%%%%%%%%%%%%%%%%%%%%
\equiv S_{i+1}[L]
=\mathcal{Z}(p \in V \,\,\, \mid \,\,\, L(p)=0 \,\,\, \text{and $p$ is nonnegative on $S_i$}).$$ The above construction eventually terminates,
i.e., for some $k\ge 0$, we have $S_{k} = S_{k+1}$
 (cf. Theorem \ref{thm:stabilize} and Remark \ref{rem:gen}).
 This allows us to define the core variety of $L$, as follows:
\begin{def*}
We call the terminal set $S_k~(=S_{k+1}=S_{k+2} \dots)$ the \textit{core variety} of $L$ and denote it by $\mathcal{CV}(L)$.
\end{def*}

The following is our main result on core varieties. 

\begin{thm*}[Theorem \ref{thm:main}]
For any nonzero $L \in V^*$ exactly one of the following holds:
\begin{enumerate}
\item The core variety of $L$ is empty and $L$ does not have a representing measure.\\
\item The core variety of $L$ is non-empty, in which case, replacing $L$ by $-L$ if necessary, we
may assume that $L(\rho) > 0$ for a strictly positive $\rho \in P$.  Then $L$ has a representing measure, and the core variety of $L$  is equal to the union of all supports of finitely atomic measures representing $L$. If $S$ is a Hausdorff space, and the functions in $V$ are continuous on $S$, then $\CV(L)$ is also equal to the union of supports of all Radon measures representing $L$.
\end{enumerate}
\end{thm*}

\begin{rem*}
We observe that if $L\in V^*$ is the zero linear functional, then $\CV(L)=\emptyset$, but $L$ does have a positive representing measure: the zero measure. 
\end{rem*}

\begin{rem*}
In part (2) of the above Theorem, concerning the case when the functions in $V$ are continuous, the condition needed in the proof is that  measures $\mu$ satisfy $\mu(S\setminus \operatorname{supp} \mu)=\emptyset$. This is automatically satisfied if $S$ is Hausdorff and $\mu$ is a \textit{Radon measure} on $S$ \cite[Appendix A.1]{Sch2}\cite[p. 155]{Bau}. Therefore we can formulate part (2) as:  If the functions in $V$ are continuous on $S$, then $\CV(L)$ is also equal to the union of supports of all measures representing $L$ and satisfying $\mu(S \setminus \operatorname{supp} \mu)=\emptyset$.
\end{rem*}

As a corollary of the above Theorem  we prove a 
version of the Bayer-Teichmann Theorem 
\cite[Theorem 2]{BT} on multivariable cubature. 
The result in \cite{BT} generalizes the classical cubature existence theorem of Tchakaloff \cite{Tch} for the polynomial case when $S$ is compact.
\begin{corollary*}[Bayer-Teichmann Theorem, Corollary \ref{cor:BT}]
Suppose that $L\in V^*$ has a representing measure $\mu$. Then $L$ has a finitely atomic representing measure $\nu$ with $\operatorname{card}\operatorname{supp} \nu \le \dim ~V$. If, additionally, $S$ is Hausdorff and $\mu$ is a Radon measure, then $\operatorname{supp} \nu \subseteq \operatorname{supp} \mu$.\end{corollary*}
The core variety can be used to describe the 
\textit{facial decomposition} of the cone $M$ of functionals in 
$V^{*}$ having representing measures. We use $\operatorname{relint}(K)$ to denote the \textit{relative interior} of a 
convex set $K$ (cf. \cite[Part II, Section 6]{rock}).

\begin{thm*}[Theorem \ref{prop:facial}]
Let $L \in M$. Let $F_L$ be the set of all linear functionals in $M$ whose core variety is contained in the core variety of $L$:
$$F_L=\{m \in M \,\,\, \mid \,\,\, \mathcal{CV}(m) \subseteq \mathcal{CV}(L)  \}.$$
Then $F_L$ is a face of $M$, and $m \in \operatorname{relint} F_L$ if and only if $\mathcal{CV}(m)=\mathcal{CV}(L)$. Furthermore, any face of $M$ has the form $F_L$ for some $L\in M$.
\end{thm*}

Results of  \cite{F2} illustrate some cases where the iterative construction of $\mathcal{CV}(L)$ can be
carried out explicitly in examples, but in general it is difficult to compute the core variety.
The following two results do, however, shed light on the behavior of the iterative construction.
The first shows a geometric termination criterion for the iterative construction of the core variety for functionals in $M$.

\begin{proposition*}[Proposition \ref{prop:faces}]
Let $L\in V^*$ be a linear functional with a representing measure. Then $\CV(L)=S_1\,\,\,(=\mathcal{Z}(p \in P \,\,\, \mid \,\,\, L(p)=0))$ if and only if $F_L$ is an exposed face of $M$.
\end{proposition*}

We next apply the preceding result to the case where some iterate $S_{k}$ is finite.

 \begin{proposition*}[Proposition \ref{prop:term}]
 Suppose that $S_k$ is a finite set of points. Then $\mathcal{CV}(L)=S_k$ or $\mathcal{CV}(L)=S_{k+1}$, i.e., the iterative construction of the core variety is guaranteed to terminate in at most one more step.
 \end{proposition*}

We prove a general version of the Truncated Riesz-Haviland Theorem in Theorem \ref{thm:trh} and derive some applications for the Full Moment Problem in a general setting. For the subsequent results, $S$ is a $\sigma$-compact, locally compact metric space; see  \cite[Chapter V, Appendix C]{Co}  for basic definition and properties.
Let $V$ be a linear subspace of the space $C(S)$ of continuous real-valued functions on $S$ which admits a decomposition as a countable union of increasing finite dimensional subspaces, i.e., $ V = \bigcup_{j=1}^{\infty} V_{j}$,
$V_{j-1}\subseteq V_{j}$, $\dim V_{j} < \infty$ ($j\ge 1$).  A function $f : S \rightarrow \RR$ \textit{vanishes at infinity} if, for each $\varepsilon >0$, there is a compact set $C_\varepsilon \subseteq S$ such that $S\setminus C_\varepsilon \subseteq  \{x \in S \mid |f(x)| < \varepsilon \}$. Let $C_{0}(S)$ denote the Banach space of all continuous functions on $S$ which vanish at infinity, equipped with the norm $|| f ||_{\infty} := \operatorname{sup}_{x\in S} |f(x)|$ (cf. \cite[p. 65]{Co}). We next present an extension of the Riesz-Haviland Theorem to our general setting.
\begin{thm*}[Theorem \ref{rh}] Suppose $1\in V_{1}$ and suppose that for each $j\ge 2$, there exists a strictly positive function $\rho_{j} \in V_{j}$
such that $\frac{q}{\rho_j}\in C_{0}(S)$ for all $q\in V_{j-1}$.
A linear functional $L:V\rightarrow \mathbb{R}$ has a representing measure if and only if
$L(p)\geq 0$ for any function $p\in V$ nonnegative on $S$.
\end{thm*}
To recover the Riesz-Haviland Theorem from this result, let $S = \mathbb{R}^{n}$, 
          $V = \mathbb{R}[x_{1},\ldots,x_{n}]$, and for $j\ge 1$, let $V_{j} = \mathcal{P}_{2j}$ and
            $\rho_{j} = 1 + ||x||^{2j}$.

Let $L$ be a linear functional on $V$. We now define the
\textit{core variety} of $L$ by $\mathcal{CV}(L) := \bigcap_{j=1}^{\infty} \mathcal{CV}(L_{j})$.
We conclude with an application of the core variety to the Full Moment Problem.
\begin{thm*}[Theorem \ref{gcv}]
A linear functional $L:V \rightarrow \mathbb{R}$ has a representing measure if and only if 
$\mathcal{CV}(L) \not = \emptyset$.
\end{thm*}
\section{Core Variety}

We begin by establishing a few preliminaries. Since the cone $P_{V}$ is pointed, the dual cone $P_V^*\subset V^*$ is full-dimensional in the dual vector space $V^*$ (\cite{barv}, p.49). 
Let $M$ be the cone of functionals in $V^*$ which have representing measures: $$M=\{L \in V^* \,\,\, \mid \,\,\, L \,\, \text{  has a representing measure}\}.$$

\noindent We observe that $M$ is a convex cone in $V^*$, but $M$ may fail to be closed {\cite{EF}.

For a point $s \in S$, let $L_s \in V^*$ denote the corresponding point evaluation functional on $V$: $L_s(f)=f(s)$. Since the points of $S$ are closed, each linear functional $L_s$ has a representing %%%%%%%%%%%%%%%%%%Borel 
measure, namely the Dirac $\delta$-measure concentrated at $s$. 
 Let $C$ be the set of all functionals coming from finitely atomic measures on $S$:
$$C=\operatorname{ConicalHull} \{L_s \, \mid \, s\in S\}.$$
%%%%%%%%%%%%%%%%%%%%%%
It is clear that $C$ is a convex cone in $V^*$ and $C \subseteq M$. Like $M$, the cone $C$ may fail to be closed, and we will later show in Corollary \ref{cor:BT} that $C=M$. Note that by Carath\'{e}odory's Theorem, 
an element of $C$ may be represented as a conical combination of
at most $\dim V$ point evaluations.

For a positive Borel measure $\mu$ on $S$, for which $V \subseteq L_{1}(\mu)$, we define the
linear functional $L_{\mu}\in V^{*}$ to be  integration with respect to $\mu$:  $L_{\mu}(f) = \int_S f d\mu$. 
Recall that the \textit{support} of $\mu$, denoted by
$\operatorname{supp}\mu$,  is defined as the set of all points $s\in S$ such that each open set $U$ containing $s$ satisfies $\mu(U) > 0$; note that $\operatorname{supp}\mu$ is closed, hence Borel. Since points of $S$ are closed, $\operatorname{supp} \mu$ is finite if and only if $\mu$ is finitely atomic, i.e., $L_\mu$ belongs to $C$.

We begin by describing the closure and interior of $C$.

\begin{proposition}\label{lemma:interior}
The closure $\overline{C}$ of $C$ is equal to the dual cone of $P$:
$$\overline{C}=P^*.$$
\end{proposition}
\begin{proof}
By the identification of $V$ with $V^{**}$ we can express the dual cone $C^*$ as
$$C^*=\{f \in V \,\,\, \mid \,\,\, L_s(f)=f(s) \geq 0 \,\,\, \text{for all} \,\,\, s \in S\}=P.$$
By the bipolarity theorem (\cite[Exercise 2.31(f)]{BV}) we have $\overline{C}=(C^*)^*$, so if follows that 
$\overline{C}=P^*$.
\end{proof}

Recall that a linear functional $L \in V^*$ is 
said to be
\textit{strictly positive} if $L(f) > 0$ whenever $f$ is a non-zero function in $P$ (cf. \cite{FN1}). 
Let $SP$ denote the convex cone of all strictly positive functionals in $V^{*}$. We next show that the interiors of $C$ and $M$ agree, and that these interiors coincide with $SP$.

\begin{cor} \label{cor:int} Let $P$ be the cone of nonnegative functions in $V$, let $M$ be the cone of linear functionals in $V^*$ which have %a 
representing measures, and let $C$ be the conical hull of the point evaluation functionals. Then
 $$\operatorname{int} (C)=\operatorname{int} (M)=
\operatorname{int}(P^*) = SP \equiv
\{L \in V^* \,\mid \, L(f)>0 \,\,\, \text{for all non-zero} \,\,\, f \in P\}.$$
 In particular, any linear functional 
%%%%%%%%%%%%%%%%%%%%%%%%%%%%%%%%%%%%%%%
$L$
strictly positive on $P$ has a finitely atomic representing measure, and the union of all supports of finitely atomic measures representing $L$ is $S$.
 \end{cor}
\begin{proof}
Note that $C %%%%%%%%%%%%%%%%%%%%%%%%%%%%%%%%%%%%
\subseteq M \subseteq P^*$, so Proposition \ref{lemma:interior} implies that 
$$P^* = \overline{C} \subseteq \overline{M} \subseteq \overline{P^*}.$$ Thus, $\overline{C} = \overline{M} = \overline{P^*}$, 
and since $C$ and $M$ are convex, 
it follows that $\operatorname{int}(C) = 
\operatorname{int}(\overline{C}) = \operatorname{int}(P^*) = \operatorname{int}(\overline{M}) = \operatorname{int}(M)$. The fact that each strictly positive functional belongs to $\operatorname{int}(P^*)$ is elementary convex geometry \cite[Exercise 2.31(d)]{BV}. Conversely, suppose 
$L \in P^*$ is not strictly positive, and let $f \in P$ be a nonzero nonnegative function such that $L(f) =0$. If $s\in S$ satisfies $f(s) > 0$, then for $\epsilon > 0$, we have $(L - \epsilon L_{s})(f) = -\epsilon f(s) < 0$, so $L$ is not in $\operatorname{int}(P^*)$. Thus, $\operatorname{int}(C)$ and $\operatorname{int}(M)$ coincide with the strictly positive functionals.

We will now prove the last point. Since $P^*$ is full-dimensional, $\operatorname{int} (P^*)$ is non-empty (\cite{barv}, p.49). Suppose that $L \in \operatorname{int} (P^*)$. Then, we see that for any $s \in S$ there exists $\epsilon >0$ such that 
%%%%%%%%%%%%%%%%%%%%%%%%%%%%%%%%%%%%%%%%%%%%%%%%%%%%%
$L':=L-\epsilon L_s \in \operatorname{int} (P^*) = \operatorname{int}(C)$. 
%Then 
Thus $L'$ has a finitely atomic representing measure, %and
whence
 $L=L'+\epsilon L_s$ has a finitely atomic representing measure whose support includes $s$. 
\end{proof}

 Given a linear functional $L \in V^*$,
 we define 
the following iterative construction: Let $S_0 \equiv S_0[L] =S$ and let $S_1\equiv S_1[L]$ be the zero set of all nonnegative functions in the kernel of $L$:
$$S_1=\mathcal{Z}(p\in P \,\,\, \mid \,\,\, L(p)=0).$$

\noindent Then we iteratively define: $$S_{i+1} 
%%%%%%%%%%%%%%%%%%%%%%%%%%%%%%%%%%%%%%%%%%%%%%%%%%%
\equiv S_{i+1}[L]
=\mathcal{Z}(p \in V \,\,\, \mid \,\,\, L(p)=0 \,\,\, \text{and $p$ is nonnegative on $S_i$}).$$
An easy induction argument shows that $S_{i+1} \subseteq S_{i}$ ($i\ge 0$). We now define the core variety of $L$ as the intersection of the nested sets $S_i$:
\begin{Def}
We call the intersection of the nested sets $S_i$ the \textit{core variety} of $L$ and denote it by $\mathcal{CV}(L)$:
$$\CV(L)=\bigcap_{i=0}^\infty S_i.$$
\end{Def}
We will show in Theorem \ref{thm:stabilize} that the iterative construction stablizes: we have $S_k=S_{k+1}=\cdots$ for a sufficiently large $k$, and, in particular, we do not need to consider infinite intersections in dealing with the core variety.

\begin{remark}\label{rem:pos}
The above results did not use the existence of a strictly positive function $\rho \in P$. It is easy to show that a strictly positive function $\rho$ exists if and only if the functions in $P$ have no common zeroes, i.e. $\mathcal{Z}(P)=\emptyset$. By construction, the core variety of any linear functional includes the set $\mathcal{Z}(P)$. Therefore, if the functions in $P$ do have common zeroes, then the core variety is never empty. Moreover, if $P$ is not full-dimensional in $V$, then for any $L\in V^*$ vanishing identically on $P$ we have $\CV(L)=\mathcal{Z}(P)$. Therefore, we are not able to determine the existence or non-existence of representing measure for such $L$ by considering only its core variety. However, with the presence of strictly positive functions we do just that in Theorem \ref{thm:main}.
\end{remark}

We start by proving in Lemma \ref{lemma:prep} (below) two basic properties of the preceding construction. 
For convenience, we will make the following standing assumption on the linear functional $L$ for all of the results up to and including Theorem \ref{thm:stabilize}.  The assumption will be removed in Theorem \ref{thm:main}, and not used in subsequent results.\\

\noindent \textbf{Standing Assumption:} \textit{There exists a strictly positive function $\rho \equiv \rho_L \in P$ such that $L(\rho)>0$.}\\
\noindent 

We remark that this assumption is standard in the Truncated Moment Problem literature. 
In the classical case when $V = \mathcal{P}_{d}$, all polynomials of degree at most $d$,
we may take $\rho$ to be the constant function $\mathbf{1}$, and the assumption $L(\mathbf{1})>0$  means that any representing measure $\mu$ satisfies $0 < \mu(S) < + \infty$; in general, since $0 < L(\rho) = \int_{S} \rho  \,d\mu$, then $\mu(S) > 0$. We note that if 
there exists a strictly positive function $f$ such that $L(f)=0$, then $S_1=\emptyset$.

\begin{lemma}\hspace{1mm}\\ \vspace{-5mm}\label{lemma:prep}
\begin{enumerate}
\item If $S_1=S$, then $L$ has a finitely atomic representing measure. The union of all supports of finitely atomic representing measures is equal to $S$.
\item If the functions in $V$ are continuous on $S$, 
then $\operatorname{supp} \mu \subseteq S_1$
for any measure $\mu$ representing $L$.
%%%%%%%%%%%%%%%%%%%%%%%%%%, we have $\operatorname{supp} \mu \subseteq S_1$.

\end{enumerate}
\end{lemma}
\begin{proof}
\begin{enumerate}
\item 
We claim that $L$ is strictly positive. Suppose that there exists a non-zero $q \in P$ with $L(q) \leq 0$. If $L(q) = 0$ and $q(s) > 0$ for some $s \in S$, then $s\not \in S_{1} =S$, a contradiction. If $L(q)<0$, then let $a:= L(\rho) > 0$, and let $b:=L(q)<0$. Then $p := \rho - \frac{a}{b}q$ is a strictly positive element of $P$ satisfying $L(p) = 0$, which implies $S_{1} = \emptyset$, a contradiction. Thus $L$ is strictly positive, so the result follows from Corollary \ref{cor:int}.\\

\item Suppose that there exists $s \in \operatorname{supp} \mu$ such that $s \notin S_1$. Then $S_1 \neq S$, and therefore there exists $p\in P$ such that $L(p)=0$ and $p(s)>0$. 
%%%%%%%%%%%%%%%%%%%%%%%%%%%%%%%Observe that 
Since $p$ is continuous,
there exists an open set $U \subset S$ containing $s$, such that $p(u)>0$ for all $u \in U$. Since $s \in \operatorname{supp} \mu$, we have $\mu(U) > 0$. It follows that $L(p)=\int_S p \, d\mu \geq \int_U p \, d\mu
 >0$, which is a contradiction. 
\end{enumerate}
\end{proof}
In the following example we show that if the functions in $V$ are merely Borel measurable,
then the conclusion of Lemma \ref{lemma:prep} part (2) need not hold.

\begin{example} \label{exm:support} Let $S = [-1,1]$ with the usual Euclidean topology, and let
$\mu$ denote the restriction of Lebesgue measure to $S$, so $\operatorname{supp}\mu = S$.
Define a function $f$  on $S$ by $f(x) = 0$ ($x\in S$, $x\not = 0$), and $f(0) = 1$.
Let $g = 1 -f$. Now $f$ and $g$ are Borel measurable, and we let $V$ denote the vector space
spanned by $f$ and $g$; note that $V$ contains $\rho := f+g = 1$ (strictly positive).
Define a functional $L$ on $V$ by $L(f) = 0$ and $L(g) = 2$; we have $L(\rho) = 2$ ($>0$).
Since the only nonnegative functions in $\operatorname{ker} L$
are scalar multiples of $f$, it follows that
 $S_{1}  = S \setminus \{0\}$ and $S_{1} = S_{2} = \cdots$ . 
It is  easy to directly check that $0$ is not
a support point of any finitely atomic representing measure for $L$. It is also easy
to verify that any two distinct nonzero points of $S$ give the support
of a $2$-atomic representing measure for $L$, so the union of these
supports is $S_{1}$. However, $\mu$ is a representing measure for $L$,
and $\operatorname{supp} \mu$ is not contained in $S_{1}$.
\end{example}
%%%%%%%%%%%%%%%%%%%%%%%%%%%%%%%%%%%%%%%%%%%%%%%%%%%%%%%%%%%%%%
We now make the following crucial observation:

\begin
{lemma}\label{lemma:consistence}
If there exists $g \in V$ such that $L(g)\neq0$ and $g$ vanishes identically on $S_1$, then there exists a function $f \in V$, strictly positive on $S_1$, such that $L(f)=0$, and therefore $S_2=\emptyset$. 
\end{lemma}
\begin{proof}

We have $a := L(\rho)>0$. We may assume $b:= L(g) < 0$, for otherwise we can consider $-g$. Then $f:= %%%%%%%%%%%%%%%%%%%%%%%%%%%%%%%%%%%%%%%%%%%%%%%%%%%%%
\rho - \frac{a}{b}g$  is strictly positive on $S_1$, and $L(f) = 0$, so $S_{2} = \emptyset$.
 
\end{proof}

We are now in position to show that the iterative construction will stabilize: $S_k=S_{k+1}$ for some $k$.

\begin{theorem}\label{thm:stabilize} For some $k \geq 0$ we have $S_k=S_{k+1}$, %%%%%%%%%%%%%%%%%%%%%%%%%%%%%%
and $L$ has a representing measure if and only if $S_k$ is non-empty.
Moreover, if $S_k$ is non-empty, then $S_k$ is the union of all supports of finitely atomic measures representing $L$. If $S$ is a Hausdorff space, and the functions in $V$ are continuous on $S$, then $\CV(L)$ is also equal to the union of supports of all Radon measures representing $L$.
\end{theorem}

\begin{rem}\label{rem:2.8}
If $L$ vanishes identically on $P$, then $\mathcal{CV}(L)=\emptyset$ in one step, since $P$ contains a strictly positive function. We also observe that according to the construction, we have $\CV(L)=\CV(-L)$. If $L$ does not vanish identically on $P$, then either $L$ or $-L$ satisfies the Standing Assumption. The Standing Assumption is simply used to pick the correct sign. This follows, since $L(p)\neq 0$ for some $p$ in the relative interior of $P$, and all functions in the relative interior of $P$ are strictly positive.  Thus, by Theorem \ref{thm:stabilize}, the iterative construction of the core variety terminates in a finite number of steps for any linear functional in $V^*$. It follows from the proof that $\CV(L)=S_k$ for $k \leq \dim V-1$.
\end{rem}
%%%%%%%%%%%%%%%%%%%%%%%%%%%%%%%%

In the proof of Theorem \ref{thm:stabilize} and elsewhere in the sequel, 
we sometimes pass from $L$ to an induced functional $\widetilde{L}$, as we next describe. 
Let $R\subset S$ and consider the subspace $W = \{ f\in V \mid f|_R \equiv 0\}$.
Let $\widetilde{V} = V/W \equiv \{f|_R \mid f\in V\}$. 
Assuming that $W \subseteq \operatorname{ker} L$, $L$ induces a well-defined element $\widetilde{L} \in
\widetilde{V}^*$ by $\widetilde{L}(f + W) \equiv \widetilde{L}(f|_R):= L(f)$ ($f\in V$). Note that if $\widetilde{L}$ admits a finitely atomic representing measure, say $\widetilde{L} = \sum a_i L_{r_{i}} (a_i > 0, r_i \in R)$, then since $L(f) = \widetilde{L}(f|_R) = \sum a_i f(r_i)$ ($f \in V$), it follows that $L$ admits a finitely atomic representing measure supported in $R$.

\begin{proof}[Proof of Theorem \ref{thm:stabilize}]

We first show that the iterative construction stabilizes.
For $i\ge 0$, let $V_{i} = \{f|_{S_{i}}\mid f\in V\} \equiv V/\{f\in V \mid f|_{S_{i}} \equiv 0\}$.
Suppose $S_{i+1}$ is nonempty and let $\tau:V_{i} \rightarrow V_{i+1}$ be the mapping given by
$\tau(f|_{S_{i}}) = f|_{S_{i+1}}$. Since $S_{i+1} \subseteq S_{i}$, $\tau$ is a well-defined linear surjection, so $\dim V_{i+1} = \dim V_{i} - \dim \ker \tau$. 
Note that if there is some $x\in S_{i} \setminus S_{i+1}$, then there exists
$f\in \ker L$ such that $f|_{S_{i}}\ge 0$ and $f(x)>0$. Thus $f|_{S_{i}}$ is a nonzero element of 
$\ker \tau$, so $\dim V_{i+1} < \dim V_{i}$. It follows that for $i\ge 0$,
either a) $S_{i} = \emptyset$, or b) $S_{i} \not = \emptyset$ and $S_{i+1} = \emptyset$,
or c) $S_{i+1} = S_{i} \not = \emptyset$, or d) $S_{i+1}$ is a proper nonempty subset of $S_{i}$,
in which case $\dim V_{i+1} < \dim V_{i}$. Since $V_{0} \equiv V$ is finite dimensional,
case d) can occur at most a finite number of times, so there is a smallest $k$ for which
$S_{k}$ satisfies a), b), or c). In cases a) and c), we have $S_{k} = S_{k+1} = \cdots$,
and in case b) we have $S_{k+1} = S_{k+2} = \cdots$, so we have established stabilization.

We now show that if $L$ has a representing measure $\mu$, then the core variety is non-empty. Indeed, we will show by induction on $i\ge 0$ that 
$S_{i}$ is a Borel set and that
$\mu(S_i)=\mu(S_0)=\mu(S)>0$. The base case $S=S_0$ is clear from the Standing Assumption, which forces $L$ to be non-zero. 
Now suppose that $S_{i}$ is a Borel set and that
 $\mu(S_i)=\mu(S)>0$. 
%%%%%%%%%%%%%%%%%%%%%%%%%%%%%%%%%%%%%%As above, we may think of $L$ as a linear %%%%%%%%%%%%%%%%%%%%%%%%%%%%%%functional on functions on $S_i$. 
Let $T$ be the span of %%%%%%%%%%%%%%%%%%%%%%%%%%%%%nonnegative 
the functions in the kernel of $L$ 
that are nonnegative on $S_{i}$. If $T = \{0\}$, then $S_{i+1} = S_{i}$, 
so the result follows by our inductive assumption.
Suppose then that $T$ is nontrivial; 
since $T$ is finite dimensional, we can choose a basis $f_1,\dots ,f_m$ of $T$ with each $f_i$ nonnegative on $S_i$. Then $S_{i+1}$ is the common zero set of $f_1,\dots,f_m$. Let $F=f_1+\dots+f_m$, so that $\mathcal{Z}(F)=\mathcal{Z}(f_1,\dots,f_m)=S_{i+1}$.
Since $F$ is Borel measurable, its zero set is Borel, so
it follows that $S_{i+1}$ is a Borel set. Now consider the complement $C$ of $S_{i+1}$ in $S_i$
(a Borel set).
We must have $\mu(C) = 0$. For otherwise,
since $\mu(S) = \mu(S_{i})$, $F |_{S_{i+1}} \equiv 0$, and $F|_C > 0$, then
$L(F)=\int_{S} F\, d\mu=\int_{S_i} F\, d\mu=\int_{C} F\, d\mu>0$, which is a contradiction to $L(F)=0$. 
%%%%%%%%%%%%%%%%%%The desired result now follows.
It follows that $\mu(S_{i+1}) = \mu(S_{i}) = \mu(S)$, as desired.

Now suppose that the core variety $\CV(L) \equiv S_{k}$ ($= S_{k+1}= \cdots$) is nonempty. 
Exactly as in the proof of Lemma \ref{lemma:consistence} (replacing $S_{1}$ by $S_{k}$ and $S_{2}$ by
$S_{k+1}$), we see that if $f\in V$ satisfies $f|_{S_{k}} \equiv 0$, then $L(f) = 0$.
Then the mapping $\widetilde{L}:V_{k} \rightarrow \mathbb{R}$ defined by 
$\widetilde{L}(f|S_{k}) := L(f)$ is a well-defined linear functional on $V_{k}$. Further,
it follows exactly as in the proof of Lemma \ref{lemma:prep}-i)  (replacing $S$ by $S_{k}$ and
$S_{1}$ by $S_{k+1}$), that $\widetilde{L}$ is strictly positive as an element of $V_{k}^{*}$. 
Corollary 2.2 thus implies that $\widetilde{L}$ has a finitely atomic 
representing measure,
and that the union of all supports of such measures is precisely $S_{k}$.
As noted in the remarks just preceding this proof,
each such measure for $\widetilde{L}$ corresponds to a finitely atomic representing measure
for $L$. We claim, conversely, that the support of each finitely atomic representing measure
$\mu$ for $L$ is contained in $S_{k}$. Indeed, since $\mu$ is a representing measure,
we know (from just above) that $\mu(S) = \mu(S_{k})$, so $\mu$ cannot have any atoms outside of
$S_{k}$. Thus $S_{k}$ is the union of all finitely atomic representing measures for $L$.

Finally, we consider the case when $S$ is Hausdorff, the measures are Radon and the functions in $V$ are continuous on $S$. For properties of Radon measures we refer to \cite[Appendix A.1]{Sch2}, \cite[p. 155]{Bau}. To complete the
proof, it suffices to show that if $\mu$ is a representing Radon measure for $L$, then $\operatorname{supp}\mu
\subseteq S_{i}$ for each $i\ge 0$, and this is clear for $i=0$. Suppose $\operatorname{supp} \mu \subseteq S_{i}$ and there exists $x\in \operatorname{supp} \mu \setminus S_{i+1}$. Since $x\in S_{i}\setminus S_{i+1}$,
there exists $f\in V$ such that $f|_{S_{i}}\ge 0$, $f(x) > 0$, and $L(f) = 0$. 
From the continuity of $f$, there is an open set $U$ such that $x\in U$ and $f|_U > 0$.
Since $x \in \operatorname{supp}\mu$, then $\mu(U) > 0$, and so $\int_{U}f d\mu > 0$. Also,
since $f|_{S_{i}} \ge 0$, then $f|_{\operatorname{supp}\mu} \ge 0$, so $\int_{\operatorname{supp}\mu} fd\mu \ge 
\int_{U \bigcap \operatorname{supp}\mu} f d\mu$.  
Now $0 = L(f) = \int_{S}f d\mu = \int_{\operatorname{supp}\mu} fd\mu \ge \int_{ U\bigcap \operatorname{supp}\mu} f d\mu 
= \int_{U} f d\mu > 0$, and this contradiction completes the proof. 

\end{proof}

The following is the main result of this section. It is a slight generalization of Theorem \ref{thm:stabilize} using the terminology of the core variety:
\begin{theorem}\label{thm:main}
For any nonzero $L \in V^*$ exactly one of the following holds:
\begin{enumerate}
\item The core variety of $L$ is empty and $L$ does not have a representing measure.\\
\item The core variety of $L$ is non-empty, in which case, replacing $L$ by $-L$ if necessary, we
may assume that $L(\rho) > 0$ for a strictly positive $\rho \in P$.  Then $L$ has a representing measure, and the core variety of $L$  is equal to the union of all supports of finitely atomic measures representing $L$. If $S$ is a Hausdorff space, and the functions in $V$ are continuous on $S$, then $\CV(L)$ is also equal to the union of supports of all Radon measures representing $L$.

\end{enumerate}
\end{theorem}
\begin{proof}
If $L$ vanishes identically on $P$, then $S_1=\emptyset$, since by our assumption $P$ contains a strictly positive function. If $L$ does not vanish on $P$, then there exists a strictly positive function $\rho \in P$, such that $L(\rho)\neq 0$ (cf. Remark \ref{rem:2.8}). Therefore, we can apply Theorem \ref{thm:stabilize} to either $L$ or $-L$ and the result follows.
\end{proof}

We now consider the number of steps before stabilization occurs in the iterative construction of the core variety.

\begin{example}

Consider the smallest value of $k$ for which $\mathcal{CV}(L) =S_{k}$.  
 In the examples of \cite{F2} and in the previous examples of this paper we have
$k\le 2$, but we now illustrate that stabilization may take $\dim V-1$ steps, which is the longest possible by the proof of Theorem \ref{thm:stabilize}.
Let $\alpha_{i} = 4i$ ($0\le i\le k+1$) and let
$\beta_{i} = x_{i} +2$ ($0\le i\le k$). Let $S$ be the set obtained from $(0,+\infty)$
by deleting $\alpha_{0},\ldots, \alpha_{k-1}$.
We define functions $f_{0},\ldots,f_{k}$ on $S$ as follows:
$$f_{0}(x) = -(x-4)\,\,\,\, \text{with} \,\,\,\, 0 < x < 4,\,\,\,\, \text{and} \,\,\,\, f_{0}(x) = 0\,\,\,\, \text{with} \,\,\,\, x > 4.$$ 
For $1\le i\le k$, we set
$f_{i}(x) := 0$  for $x \in S$ with $0< x< x_{i-1}$ or $x \ge x_{i+1}$, and
$$f_{i}(x) = -\frac{(x-\beta_{i-1})(x-\alpha_{i+1})}{(x-\alpha_{i-1})}\,\,\,\, \text{for}\,\,\,\, x_{i-1}< x < x_{i+1}.$$
Let $V$ denote the vector space spanned by $f_{0},\ldots,f_{k}$ and $\rho(x) \equiv 1$.
Since these functions are linearly independent in $C(S)$, we have $\dim V=k+2$, and we may define a linear functional $L$
on $V$ by $L(\rho) = 1$ and $L(f_{i}) = 0$ ($0\le i\le k$).
We claim that the cone $P_{0}$ of functions in $\ker L$ that are nonnegative on $S$  
consists precisely of the nonnegative multiples of $f_{0}$. To see this, suppose
$f := a_{0}f_{0} + \cdots + a_{k}f_{k}$ is nonnegative on $S$. Since $0\le f(\beta_{k}) = a_{k}f_{k}(\beta_{k})$
and $f_{k}(\beta_{k}) > 0$, we have $a_{k}\ge 0$. Now $$0\le \lim_{x \rightarrow \alpha_{k-1}^+} f(x) =
a_{k-1} \lim_{x \rightarrow \alpha_{k-1}^+} f_{k-1} + a_{k} \lim_{x \rightarrow \alpha_{k-1}^+} f_{k}(x).$$
If $a_{k}> 0$, then $\lim_{x \rightarrow \alpha_{k-1}^+} f(x) = -\infty$, and this contradiction implies that  $a_{k} = 0$.
Repeating this argument, considering the behavior of each $f_{j}$ at $\beta_{j}$ and as
$x\longrightarrow \alpha_{j-1}^+$, we may successively show that $a_{k-1} = \ldots = a_{1} = 0$, and then $a_{0} \ge 0$.
It follows that $S_{1} = \mathcal{Z}(f_{0}) = \{x\in S: x > x_{1}\}$. 
By repeating the preceding arguments, we may next show that the functions in $\ker L$ that are nonnegative on $S_{1}$
are precisely of the form $a_{0}f_{0} + a_{1}f_{1}$ with $a_{0} \in \mathbb{R}$ and $a_{1}\ge 0$, and from this it follows that
$S_{2} = \mathcal{Z}(f_{0},f_{1}) = \{x\in S: x> x_{2}\}$. In this manner, we may show successively that for $i = 1,\ldots,k$,
$P_{i} := \{f\in \ker L: f|_{S_{i}}\ge 0\} = \{a_{0}f_{0} +\cdots +a_{i-1}f_{i-1} + a_{i}f_{i}: a_{0},\cdots,a_{i-1} \in
\mathbb{R},~ a_{i}\ge 0\}$, from which it follows that $S_{i+1} = \mathcal{Z}(P_{i}) = \{x \in S: x\ge \alpha_{i+1}\}$.
Thus $S_{0},\ldots,S_{k+1}$ is a chain of strictly decreasing sets, and $\mathcal{CV}(L) =S_{k+1} =  S_{k+2} = \cdots $. Thus we have built a vector space of dimension $k+2$, where the core variety takes $k+1=\dim V-1$ steps to stabilize, which is the longest possible by the proof of Theorem \ref{thm:stabilize}.

\end{example}

\begin{remark}\label{rem:gen}
For $L \in V^{*}$, we proved Theorems \ref{thm:stabilize} and \ref{thm:main} under the assumption that $V$ contains a strictly positive function $\rho$. The argument in the proof of Theorem \ref{thm:stabilize} shows that stabilization of the core variety construction is a general phenomenon, which holds with no assumptions concerning strictly positive functions. In fact, it follows from the argument that $\CV(L)=S_k$ for $k \leq \dim V-1$. Similarly, without positivity assumptions, the proof of Theorem \ref{thm:main} shows that for $L$ nonzero,  if $L$ has a representing measure, then $\CV(L)$ is nonempty, and that if $\mu$ is a Radon representing measure, then $\operatorname{supp} \mu \subseteq \CV(L)$. For the case of Theorem \ref{thm:main} where $\CV(L)$ is nonempty, we can sometimes get by with a slightly weaker assumption than the Standing Assumption, as follows.

Suppose that $\CV(L) \equiv S_{k}$ is nonempty, and consider the following possible property for $L$: ($\mathcal{S}(L,k)$) There exist $\rho \equiv \rho_{L} \in V$ such that $\rho_{|S_{k}}$ is nonnegative, $\rho_{|S_{k}} \not \equiv 0$, and $L(\rho) \ge 0$. 
Clearly, if $L$ satisfies the Standing Assumption, then $\mathcal{S}(L,k)$ holds.

\smallskip

\noindent {\bf{Theorem.}} If $\CV(L) \equiv S_{k}$ is nonempty and $L$ satisfies $\mathcal{S}(L,k)$, then $L$ has a finitely atomic
representing measure, and the union of the supports of such measures coincides with $\CV(L)$.
\smallskip

To prove this, we adapt the proof of Theorem \ref{thm:stabilize}: condition $\mathcal{S}(L,k)$ implies that the functional
$\widetilde{L} \in V_{k}^{*}$ is well-defined and strictly positive, so the conclusion follows exactly as in the proof of
Theorem \ref{thm:stabilize}. 

The following example illustrates this approach. Let $S = [-1,2]$; let $f(0) = 1$ and $f(x) = 0$ for $x \not = 0$;
let $g(x) = 1$  for $x \in [-1,1]$ and $g(x) = -1$ for $x \in (1,2]$;
let $h(-1) = 1$ and $h(x) = 0$ for $-1 <x\le 2$.  Let $V = \langle f,~g,~h\rangle$. 
Define $L \in V^{*}$ by $L(f) = 0$, $L(g) = 2$, $L(h) = 1$. 
Then $P = \{af + bh: a,b \ge 0\}$, so there is no strictly positive element of $P$ and Theorem \ref{thm:stabilize}
does not apply. However, 
 $\CV(L) = S_{1} = S \setminus \{0\}$ and
$L$ satisfies $\mathcal{S}(L,1)$ with $\rho_{L} = h$, so the preceding theorem does apply. 
It is easy to directly check that
$\CV(L)$ is the union of supports of $2$-atomic and $3$-atomic representing measures for $L$, as predicted
by the Theorem.
\end{remark}

As a corollary of Theorem \ref{thm:main},  we next  prove a version of the Bayer-Teichmann Theorem \cite{BT} on multivariable cubature (in \cite{DS}, the result in \cite{BT} is attributed to H. Richter \cite{R}). That result
generalized the classical cubature existence theorem of Tchakaloff for the polynomial case when
$S$ is compact.
\begin{cor}[Bayer-Teichmann Theorem]\label{cor:BT}
Suppose that $L\in V^*$ has a representing measure $\mu$. Then $L$ has a finitely atomic representing measure $\nu$ with $\operatorname{card}\operatorname{supp} \nu \le \dim ~V$. If additionally, $S$ is Hausdorff and $\mu$ is a Radon measure, then $\operatorname{supp} \nu \subseteq \operatorname{supp} \mu$.\end{cor}
\begin{proof}
The first part of the Theorem follows directly from Theorem \ref{thm:main}, and  $\operatorname{card}\operatorname{supp} \nu \le \dim ~V$ follows from Carath\'{e}odory's Theorem. Now assume that $S$ is Hausdorff and $\mu$ is a Radon measure. Then we know that $\mu(S \setminus \operatorname{supp} \mu) = 0$ \cite[Appendix A.1]{Sch2}, \cite[p. 155]{Bau}. Let $\widetilde{V} := \{f|_{\operatorname{supp} \mu}: f \in V\}$. Define $\widetilde{L}$ in $\widetilde{V}^{*}$ by $\widetilde{L}(f|_{\operatorname{supp} \mu}) := L(f)$. $\widetilde{L}$ is well-defined, since if $f|_{\operatorname{supp}\mu} \equiv 0$, then  $\mu(S \setminus \operatorname{supp} \mu) = 0$ implies that $L(f) = \int f d\mu = \int_{\operatorname{supp} \mu} f d\mu = 0$. Now, $\widetilde{L}$
     has a representing measure, namely $\mu|_{\operatorname{supp}\mu}$, so Theorem \ref{thm:main} (applied with
     $S$ replaced by $\operatorname{supp} \mu$) implies that $\widetilde{L}$ has a finitely atomic representing      measure $\nu$ with $\operatorname{supp} \nu \subseteq \operatorname{supp} \mu$, and $\operatorname{card} \operatorname{supp}\nu \le
     \dim \widetilde{V} \le \dim V.$  Since $\nu$ is also a representing measure for $L$, the proof is complete.\end{proof}

\begin{remark}
For the classical polynomial case, with $V = \mathcal{P}_{d}$, Theorem \ref{thm:main}
shows that since the union of supports of all
cubature rules for $L \equiv L_{\mu}$ coincides with
$\mathcal{CV}(L)$, then this union is an
algebraic variety. This is the motivation behind the terminology ``core variety". 
\end{remark}

We next show how the core variety can be used to describe the \textit{facial decomposition} of the cone $M$ of functionals in 
$V^{*}$ having representing measures. Recall that for a general convex set $K$, a subset $F$ is a \textit{face} if $F$ is convex and
whenever $x,y \in K$ satisfy $tx + (1-t)y \in F$ for some $t\in (0,1)$, then $x,y \in F$. It is known that $K$ is the disjoint union of 
the relative interiors of its faces \cite[Thm 18.2]{rock}; moreover, for $x\in K$, the unique face to which $x$ belongs that is minimal with respect
to set inclusion (among all faces of $K$) is the face $G$ satisfying $x \in \operatorname{relint} G$ (as described above).

\begin{theorem}\label{prop:facial}
Let $L \in M$. Let $F_L$ be the set of all linear functionals in $M$ whose core variety is contained in the core variety of $L$:
$$F_L=\{m \in M \,\,\, \mid \,\,\, \mathcal{CV}(m) \subseteq \mathcal{CV}(L)  \}.$$
Then $F_L$ is a face of $M$, and $m \in \operatorname{relint} F_L$ if and only if $\mathcal{CV}(m)=\mathcal{CV}(L)$. Furthermore, any face of $M$ has the form $F_L$ for some $L\in M$.
\end{theorem}
\begin{proof}
Let $H$ be the subspace of $V$ consisting all functions vanishing on $\CV(L)$. Let $H^\perp$ be the subspace of $V^*$ consisting of functionals that are identically 0 on $H$. Let $G_L$ be the intersection of $H^\perp$ with $M$. We claim that $F_L=G_L$. Note that this implies that $F_L$ is a convex cone. Inclusion $F_L\subseteq G_L$ is immediate:  any functional in $F_L$ has a finitely atomic representing measure supported in $\CV(L)$ and therefore will be identically $0$ on $H$. The inclusion $G_L \subseteq F_L$ follows from the iterative construction of the core variety.  Indeed, let $m$ be a linear functional in $G_L$. Then $m$ has a finitely atomic representing measure $\mu$ by Corollary \ref{cor:BT}.  We claim that if the support of $\mu$ is not contained in $\CV(L)=S_k$, then it cannot be contained in $S_{k-1}$ either. Arguing as in the proof of Theorem \ref{thm:stabilize}, there exists a function $f\in V$ identically $0$ on $S_k$ and strictly positive on $S_{k-1} \setminus S_k$.  Since $f \in H $, we must have $m(f)=0$, but if the support of $\mu$ contains a point in $S_{k-1}\setminus S_k$, then $m(f)$ is positive, which is a contradiction. Repeating the same argument shows that the support of $\mu$ cannot be contained in $S_{k-1}, S_{k-2}, \dots$ until we get to $S=S_0$ and obtain a contradiction.

Next we show that $L \in \operatorname{relint} F_{L}$. Assume that $\CV(L) = S_k$, let $V_{k} =\{f|_{S_{k}}: f\in V\}$. and let $C_{k}$ denote the cone consisting of functionals in $V_{k}^{*}$ having a finitely atomic representing measure supported in $S_{k}$. For $m\in F_{L}$, 
define $\widetilde{m} \in V_{k}^{*}$ by $\widetilde{m}(f|_{S_{k}}) := m(f)$. To see that $\widetilde{m}$ is well-defined, note that from Theorem \ref{thm:main} and Corollary \ref{cor:BT}, $m$ has a finitely atomic representing measure supported in $CV(m)$ ($\subseteq CV(L)$). Thus  $m$ has the form $m = \sum a_{i} L_{s_{i}}$ with $a_{i} > 0$ and $s_{i} \in CV(m)$. Now if $f \in V$ satisfies $f|_{S_{k}} \equiv 0$, then $m(f) = \sum a_i f(s_i) = 0$, so $\widetilde{m}$ is well-defined. The preceding remarks show that $\widetilde{L}$ is well-defined. 

As in the proof of Theorem \ref{thm:stabilize}, $\widetilde{L}$ is strictly positive, so Corollary \ref{cor:int} implies that $\widetilde{L} \in \operatorname{int} C_{k}$, and, in particular, $\widetilde{L} \in \operatorname{relint} C_{k}$. Now suppose $J\in F_{L}$, so that $\widetilde{J}$ is well-defined.  For $t\in \mathbb{R}$ sufficiently close to $1$, $\widetilde{K} :=  t\widetilde{L} + (1-t)\widetilde{J}$ satisfies $\widetilde{K} \in C_{k}$. It follows that $K := tL + (1-t)J$ has a finitely atomic representing measure supported in $S_{k}$, and that if $f\in V$ satisfies $f|_{S_{k}} \equiv 0$, then $K(f) = 0$. Thus, from the above characterization of $F_{L}$,  $K \in F_{L}$, so $L\in \operatorname{relint} F_{L}$.

Now we show that $F_L$ is a face of $M$. Suppose that there exist $\ell_1, \ell_2 \in M$ such that $\ell_1+\ell_2 \in F_L$ but at least one of the functionals $\ell_i$ doesn't lie in $F_L$. Without loss of generality we may assume that $\ell_1 \notin F_L$. Then, by Theorem \ref{thm:main}, $\ell_1$ has a finitely atomic representing measure whose support is not contained in $\mathcal{CV}(L)$. Since $\ell_2$ also has a finitely atomic representing measure, we see that $\ell_1+\ell_2$ has a finitely atomic representing measure whose support is not contained in $\mathcal{CV}(L)$. This is a contradiction to Theorem \ref{thm:main} part (2), since $\ell_1+\ell_2\in F_L$. It follows that $F_L$ is a face of $M$.

Let $F$ be a face of $M$ and let $L \in \operatorname{relint} F$. Since $L$ belongs to the relative interior of a unique face of $M$ and $L \in \operatorname{relint} F_{L}$, it follows that $F = F_{L}$. If $m \in M$ satisfies $\CV(m) = \CV(L)$, then $F_{m} = F_{L}$, so $m \in \operatorname{relint} {F_m} = \operatorname{relint} F_{L}$. Conversely, suppose $m \in \operatorname{relint} F_{L}$. Since $m \in \operatorname{relint} F_{m}$, it follows by the uniqueness property in the facial decomposition that $F_{L} = F_{m}$. Thus $L \in F_{m}$ and $m\in F_{L}$ from which it follows that $\CV(L) = \CV(m)$. 
\end{proof}

To illustrate Theorem \ref{prop:facial}, we give two examples, starting with a continuation of Example \ref{exm:support}.

\begin{example} 
As in Example \ref{exm:support}, let $S = [-1,1]$ (with the Euclidean topology), and let $f$ be the
Borel measurable function on $S$ given by $f(0) = 1$ and $f(x) = 0$ ($x \neq 0$). Let $V$ be the
vector space spanned by $f$ and $g:= 1- f$. Note that for nonzero points $x$ and $y$ in $S$, we have $L_{x} = L_{y}$ in $V^*$. A basis for $V^*$ is therefore given by the functional $L$ of Example \ref{exm:support},
$L = 2L_{1}$ ($=2L_{x}$ for $x\neq 0$), together with $J:= L_{0}$. It is 
straightforward to verify that $P^*_{V} = M = \{H\equiv aL + bJ: a,b \ge 0\}$.
The four faces of $M$ are given as follows: i) $F_{J+L}$;
thus $F_{J+L} = M$, and  $H \in \operatorname{relint} F_{J+L} \longleftrightarrow 
a,b >0$, in which case $\CV(H) = S$;
ii) $F_{L}$;
thus  $H \in \operatorname{relint} F_{L} \longleftrightarrow 
a>0, ~b=0$, in which case $\CV(H) = S \setminus \{0\}$;
iii) $F_{J}$;
thus  $H \in \operatorname{relint} F_{J} \longleftrightarrow 
a=0, ~b>0$, in which case $\CV(H) = \{0\}$;
iv) $F_{\{0\}} \equiv \{0\}$, where $\CV(0) = \emptyset$.

\end{example}

\begin{example}
Let $S=\RR^n$, and let $V$ be the vector space quadratic forms (homogeneous polynomials) in $n$ variables. Then the cone $P_V$ can be identified with the cone $\mathcal{S}^n_+$ of positive semidefinite matrices. Here we recover the usual facial decomposition of $\mathcal{S}^n_+$ \cite[p. 78]{barv}: core varieties are subspaces of $\RR^n$, and for a subspace of $M\subseteq \RR^n$ the face $F$ corresponding to $M$ consists of positive semidefinite matrices with kernel contained in $M$. The relative interior of $F$ consists of positive semidefinite matrices whose kernel is equal to $M$. 
\end{example}

We next present a geometric termination criterion for the iterative construction of the core variety for functionals in $M$.

\begin{proposition}\label{prop:faces}
Let $L\in V^*$ be a linear functional with a representing measure. Then $\CV(L)=S_1=\mathcal{Z}(p \in P \,\,\, \mid \,\,\, L(p)=0)$ if and only if $F_L$ is an \normalfont{exposed face} (\cite[p. 162]{rock}) of $M$.
\end{proposition}

\begin{proof}
Let $F^\triangle_L$ be the conjugate (or dual) face of $F_L$ in $P$: $$F^\triangle_L=\{p \in P \,\,\, \mid m(p)=0 \,\,\, \text{for all $m \in F_L$}\}.$$ The minimal exposed face of $M$ containing $F_L$ is $F'_L$, which is the conjugate (dual) face of $F^\triangle_L$:
$$F'_L=\{\ell \in M \,\,\,\mid \,\,\, \ell(p)=0 \,\,\, \text{for all $p \in F^\triangle_L$}\}.$$ 
If $F_L\subsetneq F'_L$, then by Theorem \ref{prop:facial}, $F'_L$ corresponds to a strictly large core variety, and therefore the iterative construction has not terminated.

Now suppose that $\CV(L)\subsetneq \mathcal{Z}(p \in P \, \mid \, L(p)=0)$. Let $v \in \mathcal{Z}(p \in P \, \mid \, L(p)=0) \setminus \CV(L)$ and let $L_v \in V^*$ be the point evaluation functional at $v$. Consider the functional $L'=L+L_v$. From Theorem \ref{thm:main} part (2) it follows that $\{\mathcal{CV}(L) \cup v\} \subseteq \mathcal{CV}(L')$. Therefore $L'$ has a strictly larger core variety, but we also have $L' \in F'_L$. Therefore by Theorem \ref{prop:facial} we see that $F_L \subsetneq F'_L$.
 \end{proof}
 
 \begin{proposition}\label{prop:term}
 Suppose that $S_k$ is a finite set of points. Then $\mathcal{CV}(L)=S_k$ or $\mathcal{CV}(L)=S_{k+1}$, i.e. the iterative construction of the core variety is guaranteed to terminate in at most one more step.
 \end{proposition}

 \begin{proof}
 Since $S_k$ is a finite set of points, it follows that the cone $M_k$ of linear functionals with a representing measure supported on $S_k$ is simply the conical hull of point evaluations $L_s$ with $s \in S_k$. It follows that $M_k$ is a polyhedral cone, and in particular, $M_k$ is closed and all faces of $M_k$ are exposed. 
 
Suppose that $L$ has a representing measure supported on $S_k$, which is equivalent to $L \in M_k$. If $L$ lies in the interior of $M_k$ then $\mathcal{CV}(L)=S_k$ by Corollary \ref{cor:int}. If $L$ lies on the boundary of $M_k$, then $\mathcal{CV}(L)=S_{k+1}$, by Proposition \ref{prop:faces}. Finally if $L \notin M_k$ then, since $M_k$ is closed, there exists $p$ in the interior of $P_k$ such that $L(p)=0$. Such $p$ is strictly positive on $S_k$, and therefore, $\mathcal{CV}(L)=S_{k+1}=\emptyset$. 
 \end{proof}

\begin{example}

Both cases described in Proposition \ref{prop:term} can arise. In \cite[Example 3.8]{F2} we have $S = S_{0} = \mathbb{R}^{2}$ and $\CV(L) = S_{1} = S_{2} = \cdots$, with $\operatorname{card} \CV(L) = 10$. By contrast, \cite[Example 2.19]{F2} gives $S = \mathbb{R}^{2}$, $S_1$ is a set of $9$ points, and
$\CV(L) = S_{2} = S_{3} = \cdots$, with $\operatorname{card} \CV(L)  = 8$.

\end{example} 
 
\begin{remark}\label{rem:farm}
Note that the iterative construction of the core variety in terms of $S_0$, $S_1$, \ldots,
can be carried out in the setting of a general nonempty set $S$ and a linear functional $L$ acting on a finite dimensional vector space of real valued functions on $S$. All of our results concerning the connection
between the core variety and the cone of finitely atomic representing measures can then be re-formulated in terms of $C$, the conical hull of the point evaluations.         
\end{remark}        
 
\subsection{Coordinatized Perspective} A different way to think about $V$ is to pick an explicit basis $f_1, \dots, f_k$ of $V$ where $k=\dim V$. Now we can identify $V^*$ with $\RR^k$ by identifying a linear functional $L \in V^*$ with its vector of values $(L(f_1), \dots, L(f_k))$ on the basis. We will use $S^*$ to denote the set of all linear functionals $L_v$ which come from point evaluations on $S$. Define the map $\varphi_V: S\rightarrow \RR^k$ by
$$\varphi_V(v)=(f_1(v),\dots, f_k(v)).$$
Under the above identification $S^*$ is the image of $S$ under the map $\varphi_V$: $$S^*=\varphi_V(S).$$
For instance when $V$ is the vector space of polynomial functions of degree at most $d$ and $S$ is a subset of $\RR^n$ the set $S^*$ is the image of $S$ under the $d$-th Veronese map $\nu_{d}$ which sends  $x=(x_1,\dots,x_n) \in \RR^n$ to all monomials $x^\alpha$ of degree at most $d$:
$$\nu_d: \RR^n \rightarrow \RR^{\binom{n+d}{d}}, \quad \nu_d(x_1,\dots,x_n)=(1,x_1, \dots, x^{\alpha}, \dots) \quad \text{where $|\alpha| \leq d$}.$$
 
\section{Truncated Riesz-Haviland Theorem}

We now prove a generalized version of the truncated Riesz-Haviland Theorem. Our setting is as follows: Let $U$ be a subspace of a finite dimensional vector space $V$ of functions on a nonempty set $S$; we equip $V$ with Euclidean topology. Let $P_U$ and $P_V$ denote the respective cones of nonnegative functions. We will assume that $P_V$ is full-dimensional, and thus has non-empty interior in $V$, and also that it contains a strictly positive function $\rho$. Since all functions in the interior of $P_V$ are then strictly positive, we may assume that $\rho$ lies in the interior of $P_V$.  Suppose that $L$ can be \textit{extended} to a functional $\widehat{L} \in V^*$ such that $\widehat{L}$ is \textit{$V$-positive}: $$\widehat{L}(p)\geq 0 \,\,\, \text{for all}\,\,\, p\in P_V, \,\,\,\,\, \text{and} \,\,\,\,\, \widehat{L}(f)=L(f) \,\,\,\,\, \text{for all} \,\,\,\,\, f \in U.$$ We will show that with an additional hypothesis on $V$ and $U$ we may conclude that $L$ is a nonnegative linear combination of point evaluations. In the context of moment problems, this will allow us to conclude that $L$ has a representing measure.

Let $S^* \subset V^*$ be the set of linear functionals in $V^*$ corresponding to point evaluations on $S$, and let $C$ be the conical hull of $S^*$ in $V^*$. From Proposition \ref{lemma:interior}, we see that if $C$ is closed, then $P^* = \bar{C} = C \subseteq P^*$, so $P^* = C$. If the set $S^*$ is compact, then $C$ is closed  \cite[Exercise 4.17]{BPT}, and then $V$-positivity of $\widehat{L}$ already implies that $\widehat{L}$ is a nonnegative combination of point evaluations, and therefore the same holds for $L$. For the case where $S$ is a compact subset of $\mathbb{R}^{n}$ and $V = \{f|_S: f \in \mathcal{P}_{d}\}$, Tchakaloff \cite{Tch} showed that $C$ is closed. It can be shown more generally that if $S$ is a compact metric space and the functions in $V$ are continuous, then $S^{*}$ is compact in $V^{*}$.

In case $S^*$ is not compact we need to define an appropriate compactification using a strictly positive function $\rho$ in the interior of $P_V$. Define a hyperplane $H \in V^*$ consisting of functionals which evaluate to $1$ on $\rho$:
$$H=\{L \in V^* \,\,\, \mid \,\,\, L(\rho)=1\}.$$ 
Let $\widetilde{S}^*\subset H$ be the set of points $\frac{1}{\rho(y)}\ell_{y}$, with $\ell_y \in S^*$. We note that $\widetilde{S}^*$ is bounded ($\widetilde{S}^*$ lies inside the set $P_V^*\cap H$, which is a base of a closed pointed cone $P_V^*$, and is therefore compact (cf. \cite[p.66]{barv})), but it may fail to be closed. Finally, 
let $[\widetilde{S}^*]$ be the closure of $\widetilde{S}^*$ in $H$. Now we can state the Generalized Truncated Riesz-Haviland Theorem:

\begin{theorem}\label{thm:trh}Let $V$ be a finite dimensional vector space of functions on $S$ and let $U$ be a subspace of $V$. Let $L \in U^*$ be a linear functional which admits a $V$-positive extension. Suppose that for all functionals $\ell \in [\widetilde{S}^*] \setminus \widetilde{S}^*$ we have $\ell(f)=0$ for all $f \in U$. Then $L$ is a nonnegative linear combination of point evaluations.
\end{theorem}

\begin{proof}
From the preceding remarks, we may assume that 
$S^*$ is not compact. Let $M$ be the cone of representable functionals in $V^*$, and let $C$ be the conical hull of $S^*$, so that $C = \operatorname{ConicalHull}(\widetilde{S}^*)$. Recall from Proposition \ref{lemma:interior} that $\overline{C} = \overline{M} = P_{V}^*$. Since $\widetilde{S}^*$ is bounded, it follows that
$\operatorname{ConicalHull}([\widetilde{S}^*]) = \operatorname{clos}(\operatorname{ConicalHull}(\widetilde{S}^*))$, whence $P_V^*$ is simply the conical hull of $[\widetilde{S}^*]$. Therefore, we can express the extension $\widehat{L}$ as a finite  conical combination of functionals in $[\widetilde{S}^*]$:
$$ \widehat{L}=\sum_{\ell_i \in \widetilde{S}^*} \alpha_i\ell_i+\sum_{m_i \in [\widetilde{S}^*]\setminus \widetilde{S}^*} \beta_im_i,$$
where $\alpha_i,\beta_i>0$. When we restrict $\widehat{L}$ to $U$ we see that all functionals $m_i$ evaluate to $0$, and therefore the linear functional $L$ can be written as:
$$L=\sum_{\ell_i \in \widetilde{S}^*} \alpha_i\ell_i.$$
Thus $L$ is a nonnegative combination of point evaluations.

\end{proof}

\begin{exm}\label{exm:trig}
In the case when $S$ is compact, the discussion preceding Theorem \ref{thm:trh} implies that in Theorem \ref{thm:trh} we may take $U = V$ and $[\widetilde{S}^{*}] = \widetilde{S}^{*}$. Thus, if $S$ is compact, $L$ has a representing measure if and only if $L$ is $U$-positive. For example, the  Truncated Trigonometic Moment Problem on the interval $[0, 2\pi]$ falls within the scope of this case of Theorem \ref{thm:trh}.  Let $U \equiv V$ be the space of trigonometric polynomials $p$ of degree at most $n$, i.e., $p(t) = a_{0} + \sum_{j=1}^{n} [a_{j}cos(jt) + b_{j}sin(jt)]$, and let $L$ be a linear functional on $U$.  Classical results show that $L$ has a representing measure if and only if an associated Toeplitz matrix $M$ (the corresponding complex moment matrix) is positive semidefinite \cite[Theorem I.I.12]{AK}. This problem may be reformulated in terms of trigonometric polynomials on the unite circle $\mathbb{T}$.  Since, by the Fej\'{e}r-Riesz Theorem \cite{DR}, each nonnegative trigonometric polynomial on $\mathbb{T}$ of degree $n$ is the complex square of an analytic  polynomial of degree $n$, it follows that  $M$ is positive semi-definite if and only if $L$ is a positive functional, so the case of Theorem \ref{thm:trh} where $S$ is compact and the functions in $V$ are continuous is consistent with the classical solution of the Truncated Trigonometric Moment Problem.
\end{exm}
We next use Theorem \ref{thm:trh} to present a concrete realization of the Truncated Riesz-Haviland Theorem in a case where $S$ has special properties. Let $S$ be a locally compact Hausdorff space and let $C(S)$ denote the vector space of all continuous real-valued functions on $S$. A function $f : S \rightarrow \RR$ \textit{vanishes at infinity} if, for each $\varepsilon >0$, there is a compact set $C_\varepsilon \subseteq S$ such that $S\setminus C_\varepsilon \subseteq  \{x \in S \mid |f(x)| < \varepsilon \}$ (cf. \cite[p. 65]{Co}).

\begin{cor}\label{cor:trh}
Suppose that there exists $\rho \in V$ in the interior of $P_V$ such that $\frac{f}{\rho}$ vanishes at infinity for all $f \in U$. Let $L \in U^*$ be a linear functional which admits a $V$-positive extension. Then $L$ has a finitely atomic representing measure.
\end{cor}

\begin{proof}
Let $\ell$ be a linear functional 
in $[\widetilde{S}^*] \setminus \widetilde{S}^*$. Then $\ell$ is the limit of a sequence of rescaled point evaluations $\frac{1}{\rho(x_i)}\ell_{x_i}$ with $x_i \in S$. Take $f \in U$ and consider the sequence of real numbers $$\frac{1}{\rho(x_i)}\ell_{x_i}(f)=\frac{f(x_i)}{\rho(x_i)}.$$
If this sequence converges to $0$ for all $\ell \in [\widetilde{S}^{*}]\setminus \widetilde{S}^{*}$ and all $f \in U$, then we can apply Theorem \ref{thm:trh}. Suppose that there exist $\ell \in [\widetilde{S}^{*}]\setminus \widetilde{S}^{*}$ and $f \in U$ such that the above sequence of values does not converge to $0$. Then there exists a subsequence $y_i$ such that $ |\frac{f(y_i)}{\rho(y_i)}| \geq \varepsilon$ for some $\varepsilon >0$. Since $\frac{f}{\rho}$ vanishes at infinity, the points $y_i$ are contained in a compact subset of $S$, and some subsequence converges to a point $y \in S$. By continuity, it follows that $\ell=\frac{\ell_y}{\rho(y)}$ 
and therefore $\ell \in \widetilde{S}^*$, which is a contradiction.
\end{proof}

\begin{remark}
Corollary \ref{cor:trh} readily implies the polynomial case of the Truncated Riesz-Haviland Theorem \cite{CF2}. Let $P_k$ denote the vector space $n$-variate polynomials of degree at most $k$. To recover \cite{CF2} from Corollary \ref{cor:trh}, let $S \subseteq \RR^n, U = P_{2d-1}$ or $U = P_{2d-2}$, $V = P_{2d}$, and let $\rho(x) = 1 + ||x||^{2d}$; clearly, for $f \in U$, $f/\rho \in C_0(S)$.
\end{remark}

\begin{example}

The exponents of moments for the polynomial case of the Truncated Riesz-Haviland Theorem in the preceding Remark come from the dilations of the standard simplex. Using Theorem \ref{thm:trh} and Corollary \ref{cor:trh}, this can be readily extended to moments coming from other polytopes as well. For instance, let $K\subset \RR^n$ be the box $[0,a_1]\times\dots\times[0,a_n]$, with $a_i \in \NN$ and let $K'\subset \RR^n$ be the box $K'=[0,b_1]\times\dots\times[0,b_n]$, with each $b_i$ the smallest even integer strictly greater that $a_i$. Let $U$ be the vector space of $n$-variate polynomials with exponents coming from integer points of $K$, and let $V$ be the vector space of $n$-variate polynomials of with exponents coming from integer points of $K'$. Let $S\subseteq \RR^n$, and let $\rho=1+\sum_{i=1}^n x_i^{b_i}$. Then  $\rho$ is in the interior of $P_V$ and for $f \in U$, $f/\rho \in C_0(S)$, so Corollary \ref{cor:trh} applies.
\end{example}

Continuing with the setting of locally compact Hausdorff spaces, we next derive analogues of the Riesz-Haviland Theorem and of Theorem \ref{thm:main} for the generalized Full Moment Problem. Let $C_{0}(S)$ denote the Banach space of all continuous functions on $S$ which
vanish at infinity, equipped with the norm $|| f ||_{\infty} := \operatorname{sup}_{x\in S} |f(x)|$. 
The space $C_{c}(S)$ of all continuous functions on $S$ with \textit{compact support} is dense in $C_{0}(S)$.
The Riesz Representation Theorem \cite[Appendix C, Theorem 18]{Co}
states that $C_{0}(S)^{*}$, the dual space of $C_{0}(S)$, is isometrically isomorphic to $M(S)$, the Banach space of finite regular Borel measures on $S$ (equipped with the norm $||\mu|| := |\mu|(S)$); under this duality,
corresponding to $\mu \in M(S)$ is the functional $\widehat{\mu}$ on $C_{0}(S)$ 
defined by $\widehat{\mu}(f) = \int_S f d\mu$.

Recall that $B_{1}(C_{0}(S)^{*})$, the closed unit ball of $C_{0}(S)^{*}$, 
is weak-* compact (Alaoglu's Theorem) \cite[Ch. V, Theorem 3.1]{Co}. In the sequel we further assume that $S$ is 
$\sigma$-compact and metrizable (e.g., $S = \mathbb{R}^{n}$).
In  this case, $C_{0}(S)$ is separable \cite[Ch. V. Section 5, Exercise 2, page 136]{Co},
and therefore %which case 
$B_{1}(C_{0}(S)^{*})$ is a weak-* compact 
metric space \cite[Thm. V.3.1]{Co}; 
thus each bounded sequence in $C_{0}(S)^{*}$ has a weak-* convergent subsequence,
 a property that we will utilize below.

For the polynomial case and $S \subseteq \mathbb{R}^{n}$, 
in \cite{Sto}
Stochel proved that a multisequence $\beta^{(\infty)}$ has an
$S$-representing measure if and only if each truncation $\beta^{(d)}$ 
($d\ge 1$) has an
$S$-representing measure.
We next present an analogue of this result for the generalized full moment problem.

In the sequel, $S$ is a $\sigma$-compact, locally compact metric space.
Let $V$ be a linear subspace of $C(S)$ which admits a decomposition as a countable
union of increasing finite dimensional subspaces, i.e., $ V = \bigcup_{j=1}^{\infty} V_{j}$,
$V_{j-1}\subseteq V_{j}$, $\dim V_{j} < \infty$ ($j\ge 1$). By a representing measure for a linear functional $L$ on $V$ we mean a finite
 positive regular Borel measure $\mu$ on $S$ such that $L(f) = \int_{S} f d\mu$,  ($f\in V$).

\begin{theorem}\label{stochel} Suppose $1\in V_{1}$ and suppose that for each $j\ge 2$, there exists a strictly positive function $\rho_{j} \in V_{j}$
such that $\frac{q}{\rho_j}\in C_{0}(S)$ for all $q\in V_{j-1}$.
A linear functional $L:V \rightarrow \mathbb{R}$ has a representing measure if and only if each truncation
$L_{j}:= L|_{V_{j}}$ ($j\ge 1$) has a representing measure.
\end{theorem}
\begin{proof} If $\mu$ is a representing measure for $L$, then clearly 
$\mu$ is a representing measure for $L_{j}$ ($j\ge 1$).
For the converse, suppose $\mu_{j}$ is a representing measure for $L_{j}$ ($j\ge 1$).
For each $j$, $\int_{S} 1\, d\mu_{j} =  L_{j}(1) = L_{1}(1)$, so $\{\widehat{ \mu}_{j}\}_{j=1}^{\infty}$
is bounded in $C_{0}(S)^*$. It follows that a subsequence, which we also denote by
$\{\widehat{\mu}_{j}\}$, is weak-* convergent in $C_{0}(S)^*$, say to $\mu$, a
positive finite regular Borel measure on $S$. 
Let $\{q_{i}\}_{i=1}^{\infty}$ denote a basis for $V$. To complete the proof it suffices to
prove that $L(q_{i}) = \int_S q_{i} \, d\mu$ ($i\ge 1$). Fix $i\ge 1$, and let $k$ be such that
$q_{i} \in V_{k-1}$. Since the $V_{j}$'s are increasing, for $j\ge k-1$, we have $\int_S q_{i} d\mu_{j}
= L_{j}(q_{i}) = L(q_{i})$. 
Since $\frac{q_{i}}{\rho_{k}} \in C_{0}(S)$, we have
\begin{equation*}\label{colimit}
\lim_{j \rightarrow \infty}  \int_S  \frac{q_{i}}{\rho_{k}} d\mu_{j} = \int_S  \frac{q_{i}}{\rho_{k}}d \mu.
\end{equation*}
For each $f\in C_{c}(S)$, $f \cdot \rho_{k} \in C_{c}(S)$, so
\begin{equation*}\label{rho}
\lim_{j \rightarrow \infty} \int_S f\cdot\rho_{k} \,d\mu_{j} = \int_S f\cdot\rho_{k} \,d\mu.
\end{equation*}
Now $$\lim_{j \rightarrow \infty} ||\rho_{k}d\mu_{j}|| = \lim_{j \rightarrow \infty} \int_S \rho_{k}d\mu_{j} = 
\int_S \rho_{k}d\mu_{k} = L(\rho_{k}) < +\infty,$$ so $\{ \widehat{\rho_{k}d\mu_{j}}\}_{j=1}^{\infty}$ is bounded
in $C_{0}(S)^{*}$ and   convergent to $ \widehat{\rho_{k}d\mu}$ on the dense subspace
$C_{c}(S)$. Thus, $\{ \widehat{\rho_{k}d\mu_{j}}\}_{j=1}^{\infty} $ 
is weak-* convergent to $ \widehat{\rho_{k}d\mu}$ in $C_{0}(S)^{*}$.
It now follows that
$$L(q_{i}) = \int_S q_{i} d\mu_{k-1} = \lim_{j \rightarrow \infty} \int_S q_{i} d\mu_{j}=\lim_{j \rightarrow \infty} \int_S \frac{q_{i}}{\rho_{k}} (\rho_k d\mu_{j})= \int_S q_{i} d\mu.$$

\end{proof}

To recover Stochel's result from Theorem \ref{stochel}, let $S = \RR^n$, $V = \mathbb{R}[x] \equiv \mathbb{R}[x_{1},\ldots,x_{n}]$, and, for $j \ge 1$, let $V_{j} = \mathcal{P}_{2j}$ and $\rho_{j} = 1 + ||x||^{2j}$.          

\begin{exm}
Continuing Example \ref{exm:trig}, we see that in the case when $S$ is a compact metric space and the functions on $V$ are continuous, Corollary \ref{cor:trh} may be reformulated without any  requirement for the existence of  the functions $\rho_{j}$, and similarly in Theorem \ref{stochel}; equivalently, we may take each $\rho_{j} \equiv 1$. The Full Trigonometric Moment Problem, with moment functional $L$, falls within the scope of these results. It is known that a representing measure exists in this problem if and only if the associated infinite Toeplitz matrix is positive semidefinite \cite[Theorem 5.1.2]{A}, and it follows as in Example \ref{exm:trig} that this condition is equivalent to the positivity of $L$. 
\end{exm}

We next have an analogue of the Riesz-Haviland Theorem for the generalized full moment problem.

\begin{theorem}\label{rh} Let $V$ be as in Theorem \ref{stochel}.
A linear functional $L:V\rightarrow \mathbb{R}$ has a representing measure if and only if
$L$ is $V$-positive.
\end{theorem}
\begin{proof}
It is clear that if $L$ has a representing measure, then $L$ is $V$-positive.
For the converse, suppose $L$ is $V$-positive. For each $j\ge 2$, 
$U \equiv V_{j-1}$, $V \equiv V_{j}$, $\rho \equiv \rho_{j}$, and $L \equiv L_{j}$ satisfy the conditions of Corollary \ref{cor:trh}, 
which implies that $ L_{j-1}$ has a finitely atomic representing measure. 
It now follows from Theorem \ref{stochel} that $L$ has a representing measure.
\end{proof}

\begin{remark}
When $V = C_{0}(S)$, Theorem \ref{rh} is essentially equivalent to the Riesz Representation Theorem for positive linear functionals
\cite[Appendix C, Theorem 18]{Co}. This result holds for
any locally compact Hausdorff space without assuming
$\sigma$-compactness.
\end{remark}

Let $V$ be as in Theorem \ref{stochel} and let $L$ be a linear functional on $V$. 
Since $V_{j-1} \subseteq V_{j}$, it follows
that $\mathcal{CV}(L_{j}) \subseteq \mathcal{CV}(L_{j-1})$. We now define the
\textit{core variety} of $L$ by $\mathcal{CV}(L) := \bigcap_{j=1}^{\infty} \mathcal{CV}(L_{j})$. To show that $\CV(L)$ is well-defined, we show that it is independent of the decomposition of $V$ as an increasing union of finite dimensional subspaces.
Note first that if $\{V_{j_{k}}\}_{k=1}^{\infty}$ is a subsequence of $\{V_{j}\}$, then $\bigcap_{k=1}^{\infty} \CV(L|_{V_{j_{k}}}) = \bigcap_{j=1}^{\infty} \CV(L_j).$ Now suppose $V = \bigcup_{i=1}^{\infty} W_{i},$ where each $W_{i}$ is finite dimensional and 
satisfies $W_{i} \subset W_{i+1}.$ It follows from the finite dimensionality of the subspaces $W_{i} $ and $V_{j}$ that for each $k$, there exist $i_k$ and $j_k$, with $i_k < i_{k+1}, j_k < j_{k+1}$, satisfying $W_{i_{k}} \subseteq V_{j_k} \subseteq W_{i_{k+1}}$, and from this we see that $\bigcap_{i=1}^{\infty} \CV(L|_{W_{i}}) = \bigcap_{k=1}^{\infty} \CV(L|_{W_{i_{k}}})= \bigcap_{k=1}^{\infty} \CV(L|_{V_{j_{k}}}) = \bigcap_{j=1}^{\infty} \CV(L_j).$

We conclude with an analogue of Theorem \ref{thm:main}.
\begin{theorem}\label{gcv}
A linear functional $L:V \rightarrow \mathbb{R}$ has a representing measure if and only if 
$\mathcal{CV}(L) \not = \emptyset$.
\end{theorem}
\begin{proof} Let $\mu$ be a representing measure for $L$. For each $j$, $\mu$ is a representing
measure for $L_{j}$, so $\mu$ is Radon, and Theorem \ref{thm:main} (2) implies that $\operatorname{supp} \mu \subseteq \mathcal{CV}(L_{j})$. Thus, $\operatorname{supp}\mu 
\subseteq \mathcal{CV}(L)$, so $\mathcal{CV}(L)$  is nonempty. For the converse,
if $\mathcal{CV}(L) \not = \emptyset$, then for each $j$, $\mathcal{CV}(L_{j})$ is
nonempty, so Theorem \ref{thm:main} implies that $L_{j}$ has a finitely atomic representing measure. 
It now follows from Theorem \ref{stochel} that $L$ has a representing measure.
\end{proof}

\begin{example}
In the classical Full Moment Problem for $n=1$, consider a real sequence $\beta \equiv \beta^{(\infty)} = \{\beta_{i}\}_{i=0}^{\infty}$,
with Riesz functional $L_{\beta}$,
for which the corresponding Hankel matrix $M := (\beta_{i+j})_{i,j=0}^{\infty}$ is positive definite.
For each truncation $\beta^{(2d)}$, since the corresponding Hankel matrix (moment matrix) $M_{d}$ is
positive definite, \cite[Proposition 3.1]{F2} implies that $\CV(L_{\beta^{(2d)}}) = \mathbb{R}$.
Thus $\CV(L_{\beta}) = \mathbb{R}$, so Theorem \ref{gcv} implies that $L_{\beta}$ has a representing measure,
in keeping with Hamburger's Theorem (cf. \cite{A}, \cite{AK}).
\end{example}

\begin{comment}Consider the smallest value of $k$ for which $S_{k} = S_{k+1} = \cdots$. 
 In the examples of \cite{F2} and in the previous examples of this paper we have
$k\le 2$, but in general $k$ may be arbitrarily large. Here we illustrate 
$k=3$, but the example may  easily be modified to illustrate larger values of $k$.
Let $S_{0} \equiv S = (0,5)$ and define the functions $f,~g,~h \in C(S)$ as follows:
$f(x) = -\frac{1}{2}(x-2)$ ($0 < x \le 2$) and $f(x)  = 0$ ($2 <x < 5$);
$g(x) = \frac{(x-1)(x-3)}{x(x-4)}$ ($0<x\le 3$); $g(x) = 0$ ($3\le x\le 4$); $g(x)= x-4$ ($4 \le x \le 5$); 
$h(x) = f(x)$ ($0<x\le 2$); $h(x) = \frac{(x-2)(x-4)}{x-1)(x-5)}$ ($2 \le x < 5$).
Let $\rho \equiv 1$, and let $V$ be the subspace of $C(0,5)$ spannned by these functions,
which form a basis for $V$. Let $L$ be the linear functional on $V$ defined by 
$L(\rho) = 1$ and $L(f) = L(g) = L(h) = 0$. By considering the values of these functions
at $x=1$, $x=2$, $x=3$, and as $x \rightarrow 0$, we see that the nonnegative functions in 
$\ker L$ are those of the form $\alpha f$ ($\alpha \ge 0$), from which it follows that 
$S_{1} = \mathcal{Z}(f) = [2,5)$. Similarly, we see that $P_{1} := \{f\in \ker L: f|_{S_{1}} \ge 0\}= \{\alpha f + \beta g: \alpha \in \mathbb{R}, \beta \ge 0\}$, which yields $S_{2} = \mathcal{Z}(P_{1})
 = [3,4]$. Next, we see that $P_{2} := \{ f \in \ker L: f|_{S_{2}}\ge 0\} = \{\alpha f + \beta g + \gamma h:
\alpha, \beta \in \mathbb{R}, \gamma\ge 0\}$, whence $S_{3} = \{4\}$ and $\mathcal{CV}(L) = S_{3} = S_{4} = \cdots$.
\end{comment}

\end{document}